\documentclass[12pt,amsfonts]{article}
\usepackage[margin=1in]{geometry}  
\usepackage{graphicx}              
\usepackage[latin1]{inputenc}
\usepackage{amsmath, amssymb, amsthm, epsfig}               
\usepackage{enumerate}
\usepackage{amsfonts}              
\usepackage{amsthm}                
\usepackage{amsthm}
\usepackage{enumerate}
\theoremstyle{plain}
\usepackage{color}

\usepackage{tikz}  
\usepackage{pgfplots}

\def\nd{\noindent}
\def\thend{\rule{3mm}{3mm}}

\newtheorem{claim}{Claim}[section]
\newtheorem{theorem}{Theorem}[section]

\newtheorem{proposition}{Proposition}[section]
\newtheorem{lemma}{Lemma}[section]

\newtheorem{corollary}{Corollary}[section]
\newtheorem*{theorem*}{Theorem}

\numberwithin{equation}{section}

\begin{document}
\title{Multiplicity of normalized solutions for a Schr\"{o}dinger equation with critical growth in $\mathbb{R}^{N}$}
\author{ Claudianor O. Alves,\footnote{C.O. Alves was partially supported by CNPq/Brazil   grant 304804/2017-7.} \,\, Chao Ji\footnote{Corresponding author} \footnote{C. Ji was partially supported by Natural Science Foundation of Shanghai(20ZR1413900,18ZR1409100).} \,\, and \,\,   Ol\'{i}mpio H. Miyagaki\footnote{O. H. Miyagaki was supported by FAPESP/Brazil grant 2019/24901-3  and CNPq/Brazil grant 307061/2018-3.}}

\maketitle

\begin{abstract}
In this paper  we study the multiplicity of  normalized solutions  to the following nonlinear Schr\"{o}dinger equation with critical growth
\begin{align*}
 \left\{
\begin{aligned}
&-\Delta u=\lambda u+\mu |u|^{q-2}u+f(u), \quad
\quad
\hbox{in }\mathbb{R}^N,\\
&\int_{\mathbb{R}^{N}}|u|^{2}dx=a^{2},
\end{aligned}
\right.
\end{align*}
where $a,\mu>0$, $\lambda\in \mathbb{R}$ is an unknown parameter that appears as a Lagrange multiplier, $q \in (2,2+\frac{4}{N})$ and $f$ has an exponential critical growth when $N=2$, and  $f(u)=|u|^{2^*-2}u$ when $N \geq 3$ and $2^{*}=\frac{2N}{N-2}$.
\end{abstract}



\section{Introduction}

This paper concerns the multiplicity of normalized solutions to  the following nonlinear Schr\"{o}dinger equation with critical growth
\begin{align}\label{11}
 \left\{
\begin{aligned}
	&-\Delta u=\lambda u+\mu |u|^{q-2}u+f(u), \quad
	\quad
	\hbox{in }\mathbb{R}^N,\\
	&\int_{\mathbb{R}^{N}}|u|^{2}dx=a^{2},
\end{aligned}
\right.
\end{align}
where $a,\mu>0$, $\lambda\in \mathbb{R}$ is an unknown parameter that appears as a Lagrange multiplier, $q \in (2,2+\frac{4}{N})$ and $f$ has an exponential critical growth when $N=2$, and  $f(u)=|u|^{2^*-2}u$ when $N \geq 3$ and $2^{*}=\frac{2N}{N-2}$.\\

One motivation driving the search for normalized solutions
of the equation \eqref{11} is the nonlinear Schr\"{o}dinger equation
\begin{equation}\label{Tra}
i\frac{\partial \psi}{\partial t}+\triangle \psi+g(|\psi|^{2})\psi=0
\quad
\hbox{in }\mathbb{R}^N.
\end{equation}
Since the mass $\int_{\mathbb{R}^{N}}|\psi|^{2}dx$ is preserved along trajectories of \eqref{Tra}, it is natural to consider it as prescribed. A standing wave solution is a solution of the form $\psi(t, x)=e^{-i\lambda t}u(x)$, where $\lambda\in \mathbb{R}$ and $u:\mathbb{R}^N\rightarrow \mathbb{R}$ is a time-independent real valued function. This leads to \eqref{11} for $u$ with $g(|u|^{2})u=\mu |u|^{q-2}u+f(u)$.

In recent decades, the question of finding solutions of nonlinear Schr\"odinger equations with prescribed $L^{2}$-norm has  received a special attention. This  seems to be particularly meaningful from the physical point of view, because the $L^{2}$-norm is a preserved quantity of the evolution and the variational characterization of such solutions is often a strong help to analyze their orbital stability, see \cite{BellazziniJeanjeanLuo, CL, Nicola1, Nicola2} and the references therein.

A solution $u$ to the problem \eqref{11} with  $\int_{\mathbb{R}^{N}}|u|^{2}dx=a^{2}$ corresponds to a critical point of the following $C^{1}$ functional
$$
J(u)=\frac{1}{2}\int_{\mathbb{R}^N}|\nabla u|^2 \,dx-\frac{\mu}{q}\int_{\mathbb{R}^N} |u|^q \,dx -\int_{\mathbb{R}^N}F(u)\,dx,\,\, u\in H^{1}(\mathbb{R}^N),
$$
restricted to the sphere in $L^2(\mathbb{R}^N)$ given by
$$
S(a)=\{u \in H^{1}(\mathbb{R}^N)\,:\, | u |_2=a\, \},
$$
 where $F(t)=\int_{0}^{t}f(s)ds$ for $N=2$ and $F(t)=\frac{1}{2^*}|t|^{2^*}$ for $N\geq 3$.

We recall that the number $\bar{q}:=2+\frac{4}{N}$ is called  as the $L^2-$ critical exponent, which comes from   Gagliardo-Nirenberg inequality, (see \cite[Theorem 1.3.7, page 9]{CazenaveLivro} and plays a special role. From the variational point of view, when $f\equiv 0$ and $\mu>0$ in the problem \eqref{11}, then $J$ is bounded from below on $S(a)$ for the purely $L^2$-subcritical problem, i.e. $2<q<2+\frac{4}{N}$. Thus, for every $a, \mu>0$, a solution of  \eqref{11} can be found as a global minimizer of $J_{|S(a)}$, see \cite{Lions, Sh, Stu}. For the purely $L^2$-supercritical problem, i.e. $2+\frac{4}{N}<q<2^{*}$, $J_{|S(a)}$ is unbounded from below (and from above). For this case, Jeanjean \cite{jeanjean1} exploited a mountain pass structure and obtained the existence of one normalized solution. We refer \cite{Bartschmolle, jeanjean1, JeanjeanLu2020, Lions, Sh, T} where more general nonlinearities are considered. In the purely $L^2$-critical case, that is, $q =2+\frac{4}{N}$, the related problems were studied in \cite{Cheng, Miao}. Another interested topic involve the nonlinear Schr\"{o}dinger equations with combined nonlinearities, i.e. $f(u)=|u|^{p-2}u$ and $2<q\leq \bar{q}\leq p\leq 2^{*}$ for $p\neq q$ and $N\geq 3$ in \eqref{11}.  This kind of problems are more difficult, because the interplay between subcritical, critical and supercritical nonlinearities has deep impact on the geometry of the functional and on the existence and properties of ground states, we refer  \cite{Nicola1, Nicola2, Stefanov, TaoVisanZhang} for the related results (see also \cite{LZ} for the fractional case).

Existence of multiple normalized solutions has also been extensively investigated. Bartsch and de Valeriola in \cite{valerio}
 derived infinitely many radial solutions from a
fountain theorem type argument. Ikoma and Tanaka  \cite{Ikoma} provided
a multiplicity result by exploiting an idea related to symmetric
mountain pass theorems. In particular, due to our scope, we would like to mention \cite{JeanjeanLu} due to Jeanjean and Lu, where the authors showed the existence and multiplicity of nonradial solutions for problem of the type
\begin{align}\label{JeanjeanLu}
	\left\{
	\begin{aligned}
		&-\Delta u=\lambda u+g(u), \quad
		\quad
		\hbox{in }\mathbb{R}^N,\\
		&u>0,\,\,\, \int_{\mathbb{R}^{N}}|u|^{2}dx=a^{2},
	\end{aligned}
	\right.
\end{align}
where the nonliearity $g$ satisfies some technical conditions, and a model of nonlinearity that their results can be applied is
$$
g(t)=|t|^{p-2}t-|t|^{q-2}t, \quad t \in \mathbb{R},
$$
where $2<q<2+\frac{4}{N}<p<2^*$. In that paper, the main tools used are  the variational methods combined with genus theory. The results in this paper were extended in \cite{JeanjeanLu2020} including more nonlinearities. For further results
about multiplicity of normalized solutions for nonlinear Schr\"{o}dinger equations and systems, see \cite{BartschSaove},  \cite{CingolaniJeanjean}, \cite{Guo}, \cite{JeanjeanLu}, \cite{JeanjeanLe}, \cite{JeanjeanLe2}   \cite{MEDERSKISCHINO},  \cite{BenedettaTavaresVerzini},  \cite{WangLi} and references therein.

Recently, the authors in \cite{CCM} studied the existence of normalized solutions for a Schr\"{o}dinger equation with critical growth in $\mathbb{R}^{N}$ for $N \geq 2$. For existence result as ground state type solution or else mountain pass type solution involving the critical Sobolev growth for $N \geq 3$, we would also like to cite \cite{JeanjeanJendrejLeVisciglia, wei}. Another  natural problem is to search the multiplicity  of normalized solutions for a Schr\"{o}dinger equation with critical growth in $\mathbb{R}^{N}$ for $N\geq 3$ and $N=2$. Motivated by \cite{JeanjeanLu, CCM}, the goal of this paper is to find multiple normalized solutions for the problem \eqref{11} by using a minimax theorem found in \cite{JeanjeanLu} and truncation argument made  in \cite{GP}.

Our first main result is the following:
\begin{theorem}\label{T1} Assume that $f(t)=|t|^{2^*-2}t$ and $q \in (2,2+\frac{4}{N})$ with $N \geq 3$. Then, given $n \in \mathbb{N}$, there exist $\alpha>0$ independent of $n$ and $\mu_n=\mu(n)$ such that \eqref{11} admits at least $n$ couples $(u_{j}, \lambda_{j})\in H^{1}(\mathbb{R}^N)\times \mathbb{R}$ of weak solutions for $\mu \geq \mu_n$ and $ a \in \left(0,(\frac{\alpha}{\mu})^{\frac{1}{(1-\beta)q}}\right)$ with $\int_{\mathbb{R}^{N}}|u_j|^{2}dx=a^{2}$,  $\lambda_{j}<0$ and $J(u_j)<0$ for $j=1,2,..,n.$
\end{theorem}

The above theorem extends some of the previous result found in the literature for the scalar equation, because it is the first result that establishes the existence of many solutions when the combined power nonlinearities are of mixed type and are the form
$g(t)=\mu|u|^{q-2}u+|u|^{2^*-2}u,$ with $q\in (2,2+\frac{4}{N})$. However, it is very important to point out that a similar result can be proved by supposing that $g(t)=\mu|u|^{q-2}u+|u|^{p-2}u$ with $ p \in (2+\frac{4}{N},2^*)$, we state the following corollary.

\begin{corollary}\label{CT1}
	Assume that $f(t)=|t|^{p-2}t$ with  $ p \in (2+\frac{4}{N},2^*)$ and $q \in (2,2+\frac{4}{N})$ with $N \geq 3$. Then, given $n \in \mathbb{N}$, there exist $\alpha>0$ independent of $n$, $\mu_n=\mu(n)$ such that \eqref{11} admits at least $n$ couples $(u_{j}, \lambda_{j})\in H^{1}(\mathbb{R}^N)\times \mathbb{R}$ of weak solutions for $\mu \geq \mu_n$ and $ a \in \left(0,(\frac{\alpha}{\mu})^{\frac{1}{(1-\beta)q}}\right)$ with $\int_{\mathbb{R}^{N}}|u_j|^{2}dx=a^{2}$,  $\lambda_{j}<0$ and $J(u_j)<0$ for $j=1,2,..,n.$
\end{corollary}

Motivated by the research made in the critical Sobolev case, in this paper we also study the exponential critical growth  for $N=2$, which is a novelty for this type of problems. To the best of our knowledge, there is no any results  involving   normalizing problems  with the exponential critical growth, except in \cite{CCM}, where the authors considered the existence of normalized solutions. Now we recall that in $\mathbb{R}^2$, the natural growth restriction on the nonlinearity $f$  is given by the inequality
of Trudinger and Moser \cite{M,T}. More precisely, we say that a
function $f$ has an exponential critical growth if there exists $\alpha_0 >0$ such that
$$ \lim_{|s| \to \infty} \frac{|f(s)|}{e^{\alpha s^{2}}}=0
\,\,\, \forall\, \alpha > \alpha_{0}\quad \mbox{and} \quad
\lim_{|s| \to \infty} \frac{|f(s)|}{e^{\alpha s^{2}}}=+ \infty
\,\,\, \forall\, \alpha < \alpha_{0}.
$$
We would like to mention that the problems involving exponential critical growth have received a special attention at last years, see for example, \cite{A, AdoOM, ASS1,  Cao,DMR,DdOR, OS,doORuf} for semilinear elliptic equations, and \cite{1,AlvesGio,2,5} for quasilinear equations.

In our problem, we assume that $f$ is an odd continuous function that satisfies the following conditions:

\begin{itemize}
	\item[\rm ($f_1$)]$\displaystyle \lim_{t \to 0}\frac{|f(t)|}{|t|^{\tau}}=0$ as $t\rightarrow 0$,\, \mbox{for some}\, $\tau>3$;

	\item[\rm ($f_2$)]$$
	\lim_{|t|\rightarrow +\infty} \frac{|f(t)|}{e^{\alpha t^{2}}}
	=
	\begin{cases}
		0,& \hbox{for } \alpha> 4\pi,\\
		+\infty,& \hbox{for }  0<\alpha<4\pi;
	\end{cases}
	$$
	
	\item[\rm ($f_3$)] there exist constants $p>4$ and $C>0$ such that
$$
sgn(t) f(t) \geq \mu |t|^{p-1} , \forall t \in \mathbb{R}.
$$
where $sgn:\mathbb{R}\setminus \{0\} \to \mathbb{R}$ is given by
	$$
	sgn(t)=
	\left\{
	\begin{array}{l}
		1, \quad\, \,\,\,\mbox{if} \quad t>0,
		\mbox{}\\
		-1, \quad \mbox{if} \quad t<0.	
	\end{array}
	\right.
$$
\end{itemize}

Our main result is as follows:
\begin{theorem}\label{T2}
	Assume that $f$ satisfies $(f_1)-(f_3)$ and $a \in (0,1)$ when $N=2$. Then, given $n \in \mathbb{N}$, there exist $\alpha>0$ independent of $n$, $\mu_n=\mu(n)$ such that \eqref{11} admits at least $n$ couples $(u_{j}, \lambda_{j})\in H^{1}(\mathbb{R}^N)\times \mathbb{R}$ of weak solutions for $\mu \geq \mu_n$ and $ a \in \left(0,(\frac{\alpha}{\mu})^{\frac{1}{(1-\beta)q}}\right)$ with $\int_{\mathbb{R}^{N}}|u_j|^{2}dx=a^{2}$,  $\lambda_{j}<0$ and $J(u_j)<0$ for $j=1,2,..,n.$
\end{theorem}

The above theorem complements the study made by the authors in \cite{CCM}, because in that paper it was only considered that the existence of normalized solutions and did not involve the multiplicity of  normalized solutions.

In the proof of Theorem \ref{T1} and Theorem \ref{T2}, we shall use an abstract theorem involving genus theory proved by Jeanjean and Lu \cite[Theorem 2.1, Section 2]{JeanjeanLu}, for the convenience of readers, we shall provide more details in Section 2. On the other hand, in the proofs of  these theorems we shall work on the space $H^{1}_{rad}(\mathbb{R}^N)$, because it has the very nice compact embeddings. Moreover, by Palais'  principle of symmetric criticality, see \cite{palais}, we know that the critical points of $J$ in $H^{1}_{rad}(\mathbb{R}^N)$ are in fact critical points in whole $H^{1}(\mathbb{R}^N)$.

\vspace{0.5 cm}

\noindent \textbf{Notation:} From now on in this paper, otherwise mentioned, we use the following notations:
\begin{itemize}
	\item $B_r(u)$ is an open ball centered at $u$ with radius $r>0$, $B_r=B_r(0)$.

	\item   $C,C_1,C_2,...$ denote any positive constant, whose value is not relevant.
	
	\item  $|\,\,\,|_p$ denotes the usual norm of the Lebesgue space $L^{p}(\mathbb{R}^N)$, for $p \in [1,+\infty]$,
    $\Vert\,\,\,\Vert$ denotes the usual norm of the Sobolev space $H^{1}(\mathbb{R}^N)$.

	\item $o_{n}(1)$ denotes a real sequence with $o_{n}(1)\to 0$ as $n \to +\infty$.

\end{itemize}

\section{ A minimax theorem}

In this section, we present a minimax theorem for a class of constrained even functionals that is proved in Jeanjean and Lu \cite{JeanjeanLu}.

In order to formulate the minimax theorem, some notations are needed. Let $\mathcal{E}$ be a real Banach space with
norm $\Vert \cdot\Vert_{\mathcal{E}}$  and $\mathcal{H}$ be a real Hilbert space with inner product $(\cdot, \cdot)_{\mathcal{H}}$. In the sequel, let us identify $\mathcal{H}$ with its dual space
and assume that $\mathcal{E}$ is embedded continuously in $\mathcal{H}$. For any $m>0$, define the manifold
$$
\mathcal{M}:=\{u\in \mathcal{E}\mid (u, u)_{\mathcal{H}}=m \},
$$
which is endowed with the topology inherited from $\mathcal{E}$.

Clearly, the tangent space of $\mathcal{M}$ at a point
$u\in \mathcal{M}$ is defined by
$$
T_{u}\mathcal{M}:=\{v\in \mathcal{E}\mid (u, v)_{\mathcal{H}}=0 \}.
$$
Let $I\in C^{1}(\mathcal{E}, \mathbb{R})$, then $I_{|\mathcal{M} }$ is a functional of class $C^{1}$ on $\mathcal{M}$. The norm of the derivative of  $I_{|\mathcal{M} }$ at
any point $u\in \mathcal{M}$ is defined by
$$
\Vert I'_{|\mathcal{M} }\Vert:=\underset{\Vert v\Vert_{\mathcal{E}}\leq 1, v\in T_{u}\mathcal{M}}{\sup}\vert\langle I'(u), v\rangle \vert.
$$
A point $u\in \mathcal{M}$ is said to be a critical point of $I_{|\mathcal{M} }$ if $I'_{|\mathcal{M} }(u)=0$(or, equivalently, $\Vert I'_{|\mathcal{M} }(u)\Vert=0$). A number $c\in \mathbb{R}$ is called a critical value of $I_{|\mathcal{M} }$ if $I_{|\mathcal{M} }$ has a critical point $u\in \mathcal{M}$ such that $c=I(u)$. We say that $I_{|\mathcal{M} }$ satisfies the Palais-Smale condition at a level $c\in \mathbb{R}$, $(PS)_{c}$ for short, if any sequence
$\{u_{n}\}\subset \mathcal{M}$ with $I(u_{n})\rightarrow c$ and $\Vert I'_{|\mathcal{M} }(u_{n})\Vert\rightarrow 0$ contains a convergent subsequence.

Noting that $\mathcal{M}$ is symmetric with respect to $0\in \mathcal{E}$ and $0\not\in \mathcal{M}$, we introduce the notation of the
genus. Let $\Sigma(\mathcal{M})$ be the family of closed symmetric subsets of $\mathcal{M}$. For any nonempty set $A\in \Sigma(\mathcal{M})$,
the genus $\mathcal{G}(A)$ of $A$ is defined as the least integer $k\geq 1$ for which there exists an odd continuous
mapping $\varphi: A\rightarrow \mathbb{R}^k\backslash\{0\}$. We set $\mathcal{G}(A)=\infty$ if such an integer does not exist, and set $\mathcal{G}(A)=0$ if $A=\emptyset$. For each $k\in N$, let $\Gamma_{k}:=\{A\in \Sigma(\mathcal{M})\mid \mathcal{G}(A)\geq k \}$.

Now, we are ready to state the minimax theorem which will be used later on, which is a particular case of \cite[Theorem 2.1, Section 2]{JeanjeanLu}, and so, its proof will be omitted.
\begin{theorem}(Minimax theorem)\label{minimaxJL}
	Let $I: \mathcal{E}\rightarrow \mathbb{R}$ be an even functional of class $C^{1}$. Assume that
	$I_{|\mathcal{M} }$ is bounded from below and satisfies the $(PS)_{c}$ condition for all $c<0$, and that $\Gamma_{k}\neq \emptyset$ for each
	$k=1,2,...n$. Then the minimax values $-\infty<c_{1}\leq c_{2}\leq \cdots \leq c_{n}$ can be defined as follows:
	$$
	c_{k}:=\inf_{A\in \Gamma_{k}}\sup_{u\in A}I(u),\quad k=1,2,..,n,
	$$
	and the following statements hold.\\
	\noindent (i) $c_{k}$ is a critical value of $I_{|\mathcal{M} }$ provided $c_{k}<0$.\\
	\noindent (ii) Denote by $K^{c}$ the set of critical points of $I_{|\mathcal{M} }$ at a level $c\in \mathbb{R}$. If
	$$
	c_{k}= c_{k+1}= \cdots = c_{k+l-1}=:c<0\quad\text{for some}\,\,k, l\geq 1,
	$$
	then $\mathcal{G}(K^{c}) \geq l$. Hence, $I_{|\mathcal{M} }$ has at least $n$ critical points. \\
\end{theorem}

\section{Proof of Theorem \ref{T1}}

In the proof of Theorem \ref{T1} we will adapt for our case a truncation function found in Peral Alonso \cite[Chapter 2, Theorem 2.4.6]{peral}.

In what follows, we will consider the functional $J:H_{rad}^{1}(\mathbb{R}^N) \to \mathbb{R}$ given by
$$
J(u)=\frac{1}{2}\int_{\mathbb{R}^N}|\nabla u|^2 \,dx-\frac{\mu}{q}\int_{\mathbb{R}^N} |u|^q \,dx -\frac{1}{2^*}  \int_{\mathbb{R}^N}|u|^{2^*}\,dx,
$$
restricts to the sphere in $L^2(\mathbb{R}^N)$, given by
$$
S(a)=\{u \in H_{rad}^{1}(\mathbb{R}^N)\,:\, | u |_2=a\, \}.
$$
By the Sobolev embedding and the Gagliardo-Nirenberg inequality (see \cite[Theorem 1.3.7, page 9]{CazenaveLivro}  ), we have
\begin{equation}\label{Gagliardo}
|u|^{q}_{q} \leq C|u|^{((1-\beta )q}_{2} \vert \nabla u\vert^{\beta q}_{2} , \ \mbox{in}\ \mathbb{R}^N (N\geq 2),\ \beta=N(\frac{1}{2}-\frac{1}{q}),
\end{equation}
for some positive constant $C=C(q, N)>0.$ Hence,
\begin{eqnarray*}J(u) &\geq& \frac{1}{2}\int_{\mathbb{R}^N} |\nabla u|^2 \,dx-\frac{\mu C a^{(1-\beta)q}}{q}\left(\int_{\mathbb{R}^N}|\nabla u|^2 \,dx \right)^{\frac{\beta q}{2} }-\frac{1}{2^{*}S^{\frac{2^{*}}{2}}  }\left(\int_{\mathbb{R}^N}|\nabla u|^{2}\,dx\right)^{\frac{2^*}{2}},\\
&=& h( |\nabla u|_2),
\end{eqnarray*}
where
$$h(r)=\frac{1}{2} r^2 -\frac{\mu C a^{(1-\beta)q}}{q}r^{\beta q}-\frac{1}{2^* S^{\frac{2^*}{2}} } r^{2^*}.$$
 Recalling that $2< q< 2+\frac{4}{N},$ so that $\beta q < 2,$ then there exists $\alpha>0$ such that if $\mu a^{(1-\beta)q}<\alpha$, the function $h$ attains its positive local maximum (see Figure $(a)$).

For $0< R_0< R_1<\infty$ (those in Figure $(a)$ below), fix
$\tau:\mathbb{R}^+ \rightarrow [0,1]$ as being a nonincreasing and $C^\infty$ function that satisfies
$$
\tau(x)=\left\{ \begin{array}{rcr} 1 &\mbox{if}& x\leq R_0,\\ 0 &\mbox{if}& x \geq R_1.\end{array}\right.
$$
In the sequel, let us consider the truncated functional
$$
J_T(u)=\frac{1}{2}\int_{\mathbb{R}^N}|\nabla u|^2 \,dx-\frac{\mu}{q}\int_{\mathbb{R}^N} |u|^q \,dx -\frac{\tau(\vert \nabla u\vert_2)}{2^*}  \int_{\mathbb{R}^N}|u|^{2^*}\,dx.
$$
Thus
\begin{eqnarray*}J_T(u) &\geq& \frac{1}{2}\int_{\mathbb{R}^N} |\nabla u|^2 \,dx-\frac{\mu C a^{(1-\beta)q}}{q}\left(\int_{\mathbb{R}^N}|\nabla u|^2 \,dx \right)^{\frac{\beta q}{2} }-\frac{\tau(\vert \nabla u\vert_2)}{2^{*}S^{\frac{2^{*}}{2}}  }\left(\int_{\mathbb{R}^N}|\nabla u|^{2}\,dx\right)^{\frac{2^*}{2}},\\
&=& \overline{h}( |\nabla u|_2),
\end{eqnarray*}
where
$$\overline{h}(r)=\frac{1}{2} r^2 -\frac{\mu C a^{(1-\beta)q}}{q}r^{\beta q}-\frac{1}{2^* S^{\frac{2^*}{2}} } r^{2^*}\tau(r)\,\ \mbox{(See  Figure}\ (b)).$$
 \begin{center}

\begin{minipage}{7cm}
 \begin{center}

\begin{tikzpicture}

\draw[->] (-0.5,0) -- (3.5,0);
\draw[->] (0,-1.5) -- (0,2);
\node[below] at (3.5,0){$r$};
\node[left] at (0,2){$h(r)$};
\node[below] at (1.6,0){$R_0$};
\node[below] at (2.8,0){$R_1$};
\node[left] at (0,-0.2){$0$};
\node[below] at (1.5,-2.2){Figure $(a)$};
 \draw[black,domain=0:3,smooth]
plot({\x},{(\x)*(1.5-\x)*(\x-2.5)});
\filldraw [gray] (1.5,0) circle (2pt);
\filldraw [gray] (2.5,0) circle (2pt);
   \end{tikzpicture}
   \end{center}
 \end{minipage}\hfill
\begin{minipage}{7cm}
 \begin{center}

\begin{tikzpicture}

\draw[->] (-0.5,0) -- (3.5,0);
\draw[->] (0,-1.5) -- (0,2);
\node[below] at (3.5,0){$r$};
\node[left] at (0,2){$\overline{h}(r)$};
\node[below] at (2.1,0){$R_0$};
\filldraw [gray] (2,0) circle (2pt);
\node[left] at (0,-0.2){$0$};
\node[below] at (1.5,-2){Figure $(b)$};
 \draw[black,domain=0:2.7,smooth]
plot({\x},{(\x)^2-2*(\x)});

   \end{tikzpicture}
   \end{center}
   \end{minipage}

\end{center}
Without loss of generality, hereafter we will assume that
\begin{equation} \label{R0small}
\frac{1}{2} r^2 -\frac{1}{2^* S^{\frac{2^*}{2}} } r^{2^*} \geq 0, \quad  \forall \, r \in [0,R_0] \quad \mbox{and} \quad R_0 < S^{\frac{N}{2}}.
\end{equation}

Now we will give some important properties of $J_T$.
\begin{lemma}\label{Lemma 4.2}
\noindent (i) $J_T \in C^1(H_{rad}^{1}(\mathbb{R}^N), \mathbb{R}).$\\
\noindent (ii) If $J_T(u) \leq 0$ then $\vert\nabla u\vert_ 2 < R_0 ,$ and $J(v)=J_T(v),$ for all $ v$ in a small neighborhood  of $u$ in $H_{rad}^{1}(\mathbb{R}^N)$.\\
\noindent (iii) There exists $\mu_1 >0$  such that, if $0<\mu <\mu_1,$ then $J_T$ verifies a local Palais-Smale condition on $S(a)$ for the level $c < 0$.
\end{lemma}
\begin{proof}
\noindent {\bf (i)} and {\bf (ii)}  are trivial.\\
\noindent {Proof of (iii):} Let $(u_n)$ be a $(PS)_c$ sequence of $J_T$ restricts to $S(a)$ with $c < 0$. Then by the definition of $J_T$, we must have $|\nabla u_n |_2 < R_0$ for $n$ large enough, and so, $(u_n)$ is also a $(PS)_c$ sequence of $J$ restricts to $S(a)$ with $c < 0$. Hence,
\begin{equation} \label{gamma(a)}
	J(u_n) \to c \quad \mbox{as} \quad n \to +\infty,
\end{equation}
and
\begin{equation*} \label{der1}
	\|J|'_{S(a)}(u_n)\| \to 0 \quad \mbox{as} \quad n \to +\infty.
\end{equation*}
Setting the functional $\Psi:H^{1}(\mathbb{R}^N) \to \mathbb{R}$ given by
$$
\Psi(u)=\frac{1}{2}\int_{\mathbb{R}^N}|u|^2\,dx,
$$
it follows that $S(a)=\Psi^{-1}(\{a^2/2\})$. Then, by Willem \cite[Proposition 5.12]{Willem}, there exists $(\lambda_n) \subset \mathbb{R}$ such that
$$
||J'(u_n)-\lambda_n\Psi'(u_n)||_{H^{-1}} \to 0 \quad \mbox{as} \quad n \to +\infty.
$$
Hence,
\begin{equation} \label{EQ10}
	-\Delta u_n-\mu|u_n|^{q-2}u_n-|u_n|^{2^*-2}u_n=\lambda_nu_n\ + o_n(1) \quad \mbox{in} \quad (H_{rad}^{1}(\mathbb{R}^N))^*.
\end{equation}

Since $J_T$ is coercive on $S(a)$, we have that $(|\nabla u_n|_2)$ is bounded, from where it follows that $(u_n)$ is a bounded sequence in $H_{rad}^{1}(\mathbb{R}^N)$. Therefore, for some subsequence, there exists $u \in H_{rad}^{1}(\mathbb{R}^N)$ such that
$$
u_n \rightharpoonup u \quad \mbox{in} \quad H_{rad}^{1}(\mathbb{R}^N)
$$
and
\begin{equation} \label{convq}
	\lim_{n \to +\infty}\int_{\mathbb{R}^N}|u_n|^q\,dx=\int_{\mathbb{R}^N}|u|^q\,dx,
\end{equation}
because $q \in (2,2+\frac{4}{N})$.
\begin{claim} The weak limit $u$ is nontrivial, that is, $u \not= 0$.
\end{claim}
If we assume that $u=0$, we would have
\begin{equation} \label{qzero}
	\lim_{n \to +\infty}\int_{\mathbb{R}^N}|u_n|^q\,dx=0.
\end{equation}
From the definition of $J_T$, we know that
$$
J(u_n)=J_T(u_n) \geq \frac{1}{2}\int_{\mathbb{R}^N} |\nabla u_n|^2 \,dx-\frac{\mu}{q}\int_{\mathbb{R}^N}|u_n|^q \,dx -\frac{1}{2^{*}S^{\frac{2^{*}}{2}}  }\left(\int_{\mathbb{R}^N}|\nabla u_n|^{2}\,dx\right)^{\frac{2^*}{2}}.
$$
Then, by (\ref{R0small}),
$$
J(u_n) \geq -\frac{\mu}{q}\int_{\mathbb{R}^N}|u_n|^q \,dx, \quad \forall n \in \mathbb{N}.
$$
Taking the limit as $n \to +\infty$ in the last inequality and using (\ref{qzero}), we get
$$
0>c=\lim_{n \to +\infty}J(u_n) \geq - \frac{\mu}{q}\lim_{n \to +\infty}\int_{\mathbb{R}^N}|u_n|^q\,dx=0
$$
which is absurd.

On the other hand, using the fact that $(u_n)$ is bounded in $H_{rad}^{1}(\mathbb{R}^N)$, it follows from (\ref{EQ10}) and $(u_n)\subset S(a)$ that $(\lambda_n)$ is also a bounded sequence, then we can assume that for some subsequence $\lambda_n \to \lambda_a \in \mathbb{R}$. This limit together with (\ref{EQ10}) ensures that $u$ is a nontrivial solution of the equation
\begin{equation} \label{EQUA2}
-\Delta u -\mu|u|^{q-2}u-|u|^{2^*-2}u=\lambda_a u, \quad \mbox{in} \quad \mathbb{R}^N.
\end{equation}
By the Pohozaev's identity, we must have $\lambda_a<0$, because $u$ is nontrivial. Using  the concentration-compactness principle due to Lions \cite{Lions}, we can find an at most countable index set $\mathcal{J}$, sequences $(x_{i})\subset \mathbb{R}^{N}, ( \kappa_{i}), (\nu_{i})\subset (0, \infty)$ such that\\
\noindent $(i)$	\,\, $|\nabla u_n|^{2} \to \kappa$ weakly-$^*$ in the sense of measure \\
\noindent and \\
\noindent $(ii)$ \,\, $|u_n|^{2^*} \to \nu$ weakly-$^*$ in the sense of measure, \\
and
$$
\left\{
\begin{array}{l}
	(a)\quad  \nu=|u|^{2^*}+\sum_{j \in J}\nu_j \delta_{x_j},\\
	(b)\quad  \kappa \geq |\nabla u|^{2}+\sum_{j \in J}\kappa_j \delta_{x_j},\\
	(c)\quad S \nu_j^{\frac{2}{2^*}} \leq \kappa_j,\,\, \forall j\in \mathcal{J},\\
\end{array}
\right.
$$
where $\delta_{x_{j}}$ is the Dirac mass at the point $x_{j}$.\\
As
$$
-\Delta u_n-\mu|u_n|^{q-2}u_n-|u_n|^{2^*-2}u_n=\lambda_n u_n +o_n(1)\quad \mbox{in} \quad (H_{rad}^{1}(\mathbb{R}^N))^*,
$$
we derive that
$$
\int_{\mathbb{R}^N}\nabla u_n \nabla \phi \,dx-\lambda_n\int_{\mathbb{R}^N}u_n \phi \,dx=\mu \int_{\mathbb{R}^N}|u_n|^{q-2}u_n\phi\,dx+ \int_{\mathbb{R}^N}|u_n|^{2^*-2}u_n\phi\,dx, \quad \forall \phi \in H_{rad}^{1}(\mathbb{R}^N).
$$
Now, arguing as in \cite[Lemma 2.3]{GP}, $\mathcal{J}$ is empty or otherwise $\mathcal{J}$ is nonempty but finite. In the case that  $\mathcal{J}$ is nonempty but finite, we must have
$$
\kappa_j \geq {S^{\frac{N}{2}}}, \quad \forall j \in \mathcal{J},
$$
and so,
$$
R_0 \geq \limsup_{n \to +\infty}|\nabla u_n|_{2}^{2} \geq S^{\frac{N}{2}},
$$
which contradicts (\ref{R0small}). Then $\mathcal{J} = \emptyset$, and so,
\begin{equation} \label{localconvergence}
	u_n \to u \quad \mbox{in} \quad  L^{2^*}_{loc}(\mathbb{R}^N).
\end{equation}
\begin{claim} \label{convBR} For each $R>0$, we have
	$$
	u_n \to u \quad \mbox{in} \quad  L^{2^*}(\mathbb{R}^N \setminus B_R(0)).
	$$
\end{claim}	
\noindent Indeed, as $u_n \in H_{rad}^{1}(\mathbb{R}^N)$ for each $n\in N$, we know that
$$
|u_n(x)| \leq \frac{\|u_n\|}{|x|^{\frac{N-1}{2}}}, \quad \mbox{a.e. in} \quad \mathbb{R}^N.
$$
Since $(u_n)$ is a bounded sequence in $H_{rad}^{1}(\mathbb{R}^N)$, we obtain
$$
|u_n(x)| \leq \frac{C}{|x|^{\frac{N-1}{2}}}, \quad \mbox{a.e. in} \quad \mathbb{R}^N,
$$
and so,
$$
|u_n(x)|^{2^*} \leq \frac{C_1}{|x|^{\frac{N(N-1)}{N-2}}}, \quad \mbox{a.e. in} \quad \mathbb{R}^N.
$$
Recalling that  $\frac{C_1}{|\,\cdot\,|^{\frac{N(N-1)}{N-2}}} \in L^{1}(\mathbb{R}^N \setminus B_R(0))$ and $u_n(x) \to u(x)$ a.e. in $\mathbb{R}^N \setminus B_R(0)$, the Lebesgue's Theorem gives
$$
u_n \to u \quad \mbox{in} \quad  L^{2^*}(\mathbb{R}^N \setminus B_R(0)),
$$
showing Claim \ref{convBR}. Now, Claim \ref{convBR} combined with (\ref{localconvergence}) ensures that
$$
u_n \to u \quad \mbox{in} \quad L^{2^*}(\mathbb{R}^N).
$$
On the other hand,
$$
|\nabla u_n|_{2}^{2}-\lambda_n|u_n|_{2}^{2}=\int_{\mathbb{R}^N}(\mu|u_n|^q+|u_n|^{2^*})\,dx+o_n(1),
$$
or yet,
$$
|\nabla u_n|_{2}^{2}-\lambda_a|u_n|_{2}^{2}=\int_{\mathbb{R}^N}(\mu|u_n|^q+|u_n|^{2^*})\,dx+o_n(1).
$$
Now, using the limits
$$
u_n \to u \quad \mbox{in} \quad L^{2^*}(\mathbb{R}^N)
$$
and
$$
u_n \to u \quad \mbox{in} \quad L^{q}(\mathbb{R}^N),
$$
we obtain
$$
\lim_{n \to +\infty}\int_{\mathbb{R}^N}(\mu|u_n|^q+|u_n|^{2^*})\,dx=\int_{\mathbb{R}^N}(\mu|u|^q+|u|^{2^*})\,dx.
$$
By (\ref{EQUA2}), we know that
$$
|\nabla u|_{2}^{2}-\lambda_a|u|_{2}^{2}=\int_{\mathbb{R}^N}(\mu|u|^q+|u|^{2^*})\,dx,
$$
consequently
$$
\lim_{n \to +\infty}(|\nabla u_n|_{2}^{2}-\lambda_a|u_n|_{2}^{2})=|\nabla u|_{2}^{2}-\lambda_a|u|_{2}^{2}.
$$
As $\lambda_a<0$, the last equality yields that
$$
u_n \to u \quad \mbox{in} \quad H_{rad}^{1}(\mathbb{R}^N),
$$
implying that $|u|_{2}=a$. This establishes the desired result.
\end{proof}

Now, for $\epsilon >0$, we define the set
$$
A=\{ u \in H_{rad}^{1}(\mathbb{R}^N)\cap S(a): J_T(u)\leq-\epsilon\}\subset H_{rad}^{1}(\mathbb{R}^N),
$$
which is a closed symmetric subset of $S(a)$, because $J_T$ is even and continuous.

\begin{lemma}\label{Lemma 4.33}
Given $n \in \mathbb{N},$ there are $ \epsilon =\epsilon(n)>0$ and $\mu=\mu(n)>0$ such that $\mathcal{G}(A)\geq n,$ for any $ 0<\epsilon \leq\epsilon(n)$ and  $\mu \geq \mu_n$.
\end{lemma}

\begin{proof} Arguing as \cite{GP}, for each $n  \in \mathbb{N},$ let us consider an $n$-dimensional  subspace $E_n \subset H_{rad}^{1}(\mathbb{R}^N)$ that has a basis of the form
	$$
	\mathcal{B}=\{u_1,u_2,...,u_n\}
	$$
	that is orthogonal in $D^{1,2}(\mathbb{R}^N)$ and $L^{2}(\mathbb{R}^N)$  such that $|\nabla u_j|_2=\rho$ and $|u_j|_2=a$, and so, $\|u_j\|=\sqrt{\rho^2 + a^2}$  for $j=1, 2, \cdots, n$. Moreover, set
	$$
	\Upsilon_n=\{s_1u_1+s_2u_2+...+s_nu_n\,:\,s_1^2+s_2^2+....+s_n^2=1\}.
	$$
	It is easy to check that there exists a homomorphism between  $\Upsilon_n$ and the sphere \linebreak $B=\{(y_1,y_2,...,y_n) \in \mathbb{R}^{n} \,:\,y_1^2+y_2^2+....+y_n^2=\rho^2+a^2\}$ of $\mathbb{R}^{n}$. Therefore, by the properties of genus, we have $\mathcal{G}(\Upsilon_n)=n.$ For $0<\rho<R_0$ and $v \in \Upsilon_n$, we have $|\nabla v|_2 = \rho <R_0$, and so,
	$$
	J_T(v) =J( v) =\frac{1}{2}\rho^2-\frac{1}{2^*}\rho^{2^*}\int_{\mathbb{R}^N} \left|\frac{v}{\rho}\right|^{2^*}\, dx -\frac{\mu}{q}\rho^q \int_{\mathbb{R}^N} \left|\frac{v}{\rho}\right|^q \, dx.
	$$
	Recall that dim$E_n=n$, all the norms are equivalents and define
	$$
	\alpha_n=\inf\left\{ \int_{\mathbb{R}^N} |w|^{2^*}\, dx: w\in S(a/\rho) \cap E_n, |\nabla w|_2=1\right\} >0
	$$
	and
	$$
	\beta_n=\inf\left\{ \int_{\mathbb{R}^N} |w|^{q}\, dx: w\in S(a/\rho) \cap E_n, |\nabla w|_2=1\right\} >0.
	$$
	Then,
$$
	J_T(v) \leq  \frac{1}{2}\rho^2-\frac{1}{2^*}\rho^{2^*}\alpha_n  -\frac{\mu}{q}\rho^q \beta_n.
	$$

From this, we can choose $\mu=\mu(n)>0$, $\epsilon=\epsilon(n)>0$ and $\rho< R_0$ such that
$$
J_T(v)\leq -\epsilon, \quad \forall v \in \Upsilon_n.
$$
Hence, $\Upsilon_n \subset A$ and $\mathcal{G} (A) \geq \mathcal{G}(\Upsilon_n)=n$.

\end{proof}

\begin{proposition}\label{Lemma 4.4}  Let $\Sigma_k=\{ C\subset H^{1}_{rad}(\mathbb{R}^N)\cap S(a): C  \ \mbox{is closed}\ , C=-C\,\,\text{and}\,\,\mathcal{G}(C)\geq k\}$.
	Let $$c_k=\inf_{C\in \Sigma_k} \sup_{u\in C} J_T(u), $$
	and $K_c=\{ u \in H^{1}_{rad}(\mathbb{R}^N)\cap S(a): J^{\prime} _T(u)=0, J_T(u)=c\},$
	and suppose $0< \mu< \mu_1$, where $\mu_1$ is the constant in Lemma \ref{Lemma 4.2}.
	If $c=c_k=c_{k+1}= \ldots =c_{k+r},$ we have $\mathcal{G}(K_c)\geq r+1.$  In particular, $J_T$ has at least $k$ nontrivial critical points.
\end{proposition}
\begin{proof}
	For $\epsilon >0$, define
	$$
	A_{-\epsilon}=\{ u \in H^{1}(\mathbb{R}^N)\cap S(a): J_T(u)\leq-\epsilon\}\subset H_{rad}^{1}(\mathbb{R}^N).
	$$
	From Lemma \ref{Lemma 4.33}, for all $k \in \mathbb{N},$  there exists $\epsilon=
	\epsilon(k)>0$ such that $\mathcal{G}(A_{-\epsilon}) \geq k.$
	Since $J_T$ is continuous and even, $A_{-\epsilon}\in \Sigma_k$. Then $c_k\leq -\epsilon<0$ for any $k$.
	On the other hand, $J_T$ is  bounded from below, then  this implies that $ c_k> -\infty$ for any $k$.
	Let us assume that $c=c_k=c_{k+1}= \ldots =c_{k+r},$ notice that $c<0$, then by Lemma \ref{Lemma 4.2}, $J_T$ verifies the Palais-Smale condition at the level $c<0$, from where it follows that $K_c$ is a compact set.
	Now, the proposition follows from Theorem \ref{minimaxJL}.
\end{proof}

\subsection{Proof Theorem \ref{T1}}

The proof of Theorem \ref{T1} follows from Proposition \ref{Lemma 4.4}, because the critical points of $J_T$ that were found in that proposition are in fact critical points of $J$, see Lemma \ref{Lemma 4.2}-(ii).

\section{Proof of Theorem \ref{T2}}

In the proof of Theorem \ref{T2} we shall follow the same ideas explored in Section 3. However, since we will work with a nonlinear term which has an exponential critical growth,  some estimates are different and more difficult. In the present case, we shall apply the Trudinger-Moser inequality proved by Cao \cite{Cao} that plays an important role in our approach.
\begin{lemma}\label{Cao}(Trudinger-Moser inequality by Cao \cite{Cao})
	If $\alpha>0$ and $u\in H^{1}(\mathbb{R}^{2})$, then
	\begin{equation*}
		\int_{\mathbb{R}^{2}}(e^{\alpha  u^{2}}-1)dx<+\infty.
	\end{equation*}
	Moreover, if $|\nabla u|_{2}^{2}\leq 1$, $|  u|_{2}\leq M<+\infty$, and $0<\alpha< 4\pi$, then there exists a positive constant $C(M, \alpha)$, which depends only on $M$ and $\alpha$,  such that
	\begin{equation*}
		\int_{\mathbb{R}^{2}}(e^{\alpha u^{2}}-1)dx\leq C(M, \alpha).
	\end{equation*}
\end{lemma}

From  $(f_1)$ and $(f_2)$, we know that fixed $\varrho >2$, for any $\zeta>0$ and $\alpha>4\pi$, there exists a constant $C>0$, which depends on $q$, $\alpha$, $\zeta$, such that
\begin{equation}
	\label{1.2}
	|f(t)|\leq\zeta |t|^{\tau}+C|t|^{\varrho-1}(e^{\alpha t^{2}}-1) \text{ for all } t \in \mathbb{R}
\end{equation}
and so,
\begin{equation}
	\label{1.3}
	|F(t)|\leq\zeta |t|^{\tau+1}+C|t|^{\varrho}(e^{\alpha t^{2}}-1) \text{ for all } t \in \mathbb{R}.
\end{equation}
Moreover, it is easy to see that, by \eqref{1.2},
\begin{equation}
	\label{1.4}
	|f(t)t| \leq\zeta |t|^{\tau+1}+C\vert t\vert^{\varrho}(e^{\alpha  t^{2}}-1) \text{ for all } t\in\mathbb{R}.
\end{equation}

Next, we recall some technical lemmas that can found in \cite{CCM}, here we omit their proofs.

\begin{lemma} \label{alphat11} Let $a \in (0,1)$ and $(u_{n})$ be a sequence in $H^{1}(\mathbb{R}^{2})$ with $u_n \in S(a)$ and
	$$
	\limsup_{n \to +\infty} |\nabla u_n |_{2}^{2}  < 1-a^{2}.
	$$
	Then, there exist  $t> 1$, $t$ close to 1,  and $C > 0$  satisfying
	\[
	\int_{\mathbb{R}^{2}}\left(e^{4 \pi |u_n|^{2}} - 1 \right)^t dx \leq C, \,\,\,\,\forall\, n \in \mathbb{N}.
	\]
	
\end{lemma}

\begin{lemma} \label{convergencia} Let $a \in (0,1)$ and $(u_n) \subset H_{rad}^{1}(\mathbb{R}^{2})$ be a sequence  with $u_n \in S(a)$ and
	$$
	\limsup_{n \to +\infty} |\nabla u_n |_2^{2}  < 1-a^2.
	$$
	Then, there exists $\alpha>4\pi$ close to $4\pi$, such that for all $\varrho>2$,
	$$
	|u_n|^{\varrho}(e^{\alpha  |u_n(x)|^{2}}-1) \to |u|^{\varrho}(e^{\alpha |u(x)|^{2}}-1) \,\, \mbox{in} \,\, L^{1}(\mathbb{R}^N).
	$$
\end{lemma}
\begin{lemma} \label{Convergencia em limitados1} Let $a \in (0,1)$ and $(u_{n})$ be a sequence in $H_{rad}^{1}(\mathbb{R}^{2})$ with $u_n \in S(a)$ and
	$$
	\limsup_{n \to +\infty} |\nabla u_n |_2^{2}  < 1-a^2.
	$$
	If $u_n \rightharpoonup u$ in $H^{1}(\mathbb{R}^{2})$ and $u_n(x) \to u(x)$ a.e in $\mathbb{R}^{2}$, then
	$$
	F(u_n) \to F(u) \quad \mbox{and} \quad f(u_n)u_n \to f(u)u \,\, \mbox{in} \,\, L^{1}(\mathbb{R}^2).
	$$ 	
\end{lemma}

In what follows, we will consider the functional $J:H_{rad}^{1}(\mathbb{R}^2) \to \mathbb{R}$ given by
$$
J(u)=\frac{1}{2}\int_{\mathbb{R}^2}|\nabla u|^2 \,dx-\frac{\mu}{q}\int_{\mathbb{R}^2} |u|^q \,dx -\int_{\mathbb{R}^2}F(u)\,dx,
$$
restricts to the sphere in $L^2(\mathbb{R}^2)$, given by
$$
S(a)=\{u \in H_{rad}^{1}(\mathbb{R}^2)\,:\, | u |_2=a\, \},
$$
where $a\in (0, 1)$. By the Sobolev embedding and the Gagliardo-Nirenberg inequality (see \cite[Theorem 1.3.7, page 9]{CazenaveLivro}  ),  we have
\begin{equation}\label{Gagliardo}
	|u|^{q}_{q} \leq C|u|^{2}_{2} \vert \nabla u\vert^{\beta q}_{2} , \ \mbox{in}\ \mathbb{R}^2,\, \beta=1-\frac{2}{q},
\end{equation}
for some positive constant $C=C(\varrho, 2)>0.$ By \eqref{1.3},
\begin{equation*}
	|F(t)|\leq\zeta |t|^{\tau+1}+C|t|^{\varrho}(e^{\alpha t^{2}}-1) \text{ for all } t \in \mathbb{R}.
\end{equation*}
where $\varrho>2$ and $\alpha>4\pi$ as in the last Lemma \ref{convergencia}. Hence,
$$
|F(u)| \leq \zeta|u|^{\tau+1}+C|u|^{\varrho}(e^{\alpha |u|^{2}}-1).
$$

Using the H\"older's inequality, one has
$$
\int_{\mathbb{R}^2}|F(u)|\,dx \leq \zeta\int_{\mathbb{R}^2}|u|^{\tau+1}\,dx+C\Big(\int_{\mathbb{R}^2}|u|^{\varrho t}\,dx\Big)^{1/t}\Big(\int_{\mathbb{R}^2}(e^{\alpha |u|^{2}}-1)^{t'}dx\Big)^{1/t'},
$$
 where $t$ is the conjugate exponent of $t'$ and $t'>1$ is close to 1. Arguing as in \cite[Lemma 4.2]{CCM}, there exists $C=C(u,m)>0$ such that
 $$
 \int_{\mathbb{R}^2}(e^{\alpha |u|^{2}}-1)^{t'}dx\leq C.
 $$

 Hence,
$$
\begin{array}{l}
	J(u) \geq  \displaystyle  \frac{1}{2}\int_{\mathbb{R}^2} |\nabla u|^2 \,dx-\frac{\mu C a^{2}}{q}\left(\int_{\mathbb{R}^2}|\nabla u|^2 \,dx \right)^{\frac{q-2}{2} }\\
\mbox{} \\
\hspace{1.2 cm} 	\displaystyle -C_1\left(\int_{\mathbb{R}^2}|\nabla u|^2 \,dx \right)^{\frac{\tau-1}{2} }-C_2\left(\int_{\mathbb{R}^2}|\nabla u|^2 \,dx \right)^{(\frac{\varrho}{2}-\frac{1}{2t'}) }\\
\mbox{}\\
\hspace{1.2 cm} \geq 	h( |\nabla u|_2),
\end{array}
$$
where
$$h(r)=\frac{1}{2} r^2 -\frac{\mu C a^2}{q}r^{(q-2)}-C_1 r^{(\tau-1)}-C_2 r^{(\varrho-\frac{1}{t'})}.$$
Recalling that $2< q< 2+\frac{4}{N},$ so that, $q-2 < 2$ since $N\geq 3$, while $\tau-1>2$, $(\varrho-\frac{1}{t'}) >2$.
Then, there exists $\alpha>0$ such that if $\mu a^{2}<\alpha$, $h$ attains its positive local maximum, as in Section 3. For $0< R_0< R_1<\infty,$ let us fix
$\tau:\mathbb{R}^+ \rightarrow [0,1]$ as being a nonincreasing and $C^\infty$ function that satisfies
$$
\tau(x)=\left\{ \begin{array}{rcr} 1 &\mbox{if}& x\leq R_0,\\ 0 &\mbox{if}& x \geq R_1.\end{array}\right.
$$
In the sequel, we consider the truncated functional
$$
J_T(u)=\frac{1}{2}\int_{\mathbb{R}^2}|\nabla u|^2 \,dx-\frac{\mu}{q}\int_{\mathbb{R}^2} |u|^q \,dx -\tau(\vert \nabla u\vert_2)\int_{\mathbb{R}^2}F(u)\,dx,
$$
Thus,
$$
\begin{array}{l}
	J_T(u) \geq  \displaystyle  \frac{1}{2}\int_{\mathbb{R}^2} |\nabla u|^2 \,dx-\frac{\mu C a^{2}}{q}\left(\int_{\mathbb{R}^2}|\nabla u|^2 \,dx \right)^{\frac{q-2}{2} }\\
\mbox{} \\
\hspace{1.2 cm} 	\displaystyle -\tau(\vert \nabla u\vert_2)\Big(C_1\left(\int_{\mathbb{R}^2}|\nabla u|^2 \,dx \right)^{\frac{\tau-1}{2} }+C_2\left(\int_{\mathbb{R}^2}|\nabla u|^2 \,dx \right)^{(\frac{\varrho}{2}-\frac{1}{2t'}) }\Big)\\
\mbox{}\\
\hspace{1.2 cm} =\overline{h}(u),
\end{array}
$$
where
$$
\overline{h}(r)=\frac{1}{2} r^2 -\frac{\mu C_qa^{2}}{q}r^{(q-2)}-(C_1 r^{(\tau-1)}+C_2 r^{(\varrho-\frac{1}{t'})})\tau(r).
$$
Without loss of generality, hereafter we will assume that
\begin{equation} \label{R0}
	\frac{1}{2} r^2 -(C_1 r^{(\tau-1)}+C_2 r^{(\varrho-\frac{1}{t'})}) \geq 0, \quad  \forall r \in [0,R_0] \quad \mbox{and} \quad R_0 < \sqrt{1-a^2}.
\end{equation}

\begin{lemma}\label{Lemma 4.2N2}
	\noindent (i) $J_T \in C^1(H_{rad}^{1}(\mathbb{R}^2), \mathbb{R}).$\\
	\noindent (ii) If $J_T(u) \leq 0$ then $\vert\nabla u\vert_ 2 < R_0 ,$ and $J(v)=J_T(v),$ for all $ v$ in a small neighborhood  of $u$ in $H^{1}(\mathbb{R}^2)$.\\
	\noindent (iii) There exists $\mu_1 >0$  such that, if $0<\mu <\mu_1,$ then $J_T$ verifies a local Palais-Smale condition on $S(a)$ with $a\in (0, 1)$  for the level $c < 0$.
\end{lemma}
\begin{proof}
	\noindent {\bf (i)} and {\bf (ii)} are trivial. \\
	\noindent {Proof of (iii):} Let $(u_n)$ be a $(PS)_c$ sequence of $J_T$ restricts to $S(a)$ with $c < 0$. Then by the definition of $J_T$, we must have $|\nabla u_n|_2 < R_0$, and so, $(u_n)$ is also a $(PS)_c$ sequence of $J$ restricts to $S(a)$ with $c < 0$. Hence,
	\begin{equation} \label{gamma(a)}
		J(u_n) \to c \quad \mbox{as} \quad n \to +\infty,
	\end{equation}
	and
	\begin{equation*} \label{der1}
		\|J|'_{S(a)}(u_n)\| \to 0 \quad \mbox{as} \quad n \to +\infty.
	\end{equation*}
	Setting the functional $\Psi:H^{1}(\mathbb{R}^2) \to \mathbb{R}$ given by
	$$
	\Psi(u)=\frac{1}{2}\int_{\mathbb{R}^2}|u|^2\,dx,
	$$
	it follows that $S(a)=\Psi^{-1}(\{a^2/2\})$. Then, by Willem \cite[Proposition 5.12]{Willem}, there exists $(\lambda_n) \subset \mathbb{R}$ such that
	$$
	||J'(u_n)-\lambda_n\Psi'(u_n)||_{H^{-1}} \to 0 \quad \mbox{as} \quad n \to +\infty.
	$$
	Hence,
\begin{equation} \label{EQ10N2}
		-\Delta u_n-\mu|u_n|^{q-2}u_n-f(u_n)u_n=\lambda_nu_n\ + o_n(1) \quad \mbox{in} \quad (H_{rad}^{1}(\mathbb{R}^2))^*.
	\end{equation}
	
	Since $J_T$ is coercive on $S(a)$, we have that $(|\nabla u_n|_2)$ is bounded, from where it follows that $(u_n)$ is a bounded sequence in $H_{rad}^{1}(\mathbb{R}^2)$. Therefore, for some subsequence, there is $u \in H_{rad}^{1}(\mathbb{R}^2)$ such that
	$$
	u_n \rightharpoonup u \quad \mbox{in} \quad H^{1}(\mathbb{R}^2)
	$$
	and
	\begin{equation} \label{convq}
		\lim_{n \to +\infty}\int_{\mathbb{R}^2}|u_n|^q\,dx=\int_{\mathbb{R}^2}|u|^q\,dx,
	\end{equation}
	because $q \in (2,2+\frac{4}{N})$.
	\begin{claim} The weak limit $u$ is nontrivial, that is, $u \not= 0$.
	\end{claim}
	If we assume that $u=0$, we would have
	\begin{equation} \label{q}
		\lim_{n \to +\infty}\int_{\mathbb{R}^2}|u_n|^q\,dx=0.
	\end{equation}
Moreover, since $|\nabla u_n|<R_0$ and $R_0 < \sqrt{1-a^2}$, we can apply Lemma \ref{Convergencia em limitados1} to get
\begin{equation} \label{F}
	\lim_{n \to +\infty}\int_{\mathbb{R}^2}F(u_n)\,dx=0
\end{equation}
	From the definition of $J_T$, we know that
	$$
	J(u_u)=J_T(u_n)= \frac{1}{2}\int_{\mathbb{R}^2} |\nabla u_n|^2 \,dx-\frac{\mu}{q}\int_{\mathbb{R}^2}|u_n|^q \,dx -\int_{\mathbb{R}^2}F(u_n)\,dx.
	$$
	Then,
	$$
c=\lim_{n \to +\infty}J(u_n) \geq -\frac{\mu}{q}\int_{\mathbb{R}^2}|u_n|^q \,dx-\int_{\mathbb{R}^2}F(u_n)\,dx, \quad \forall n \in \mathbb{N}.
	$$
Taking the limit as $n \to +\infty$ in the last inequality and using (\ref{q})-(\ref{F}), we get
	$$
	c=\lim_{n \to +\infty}J(u_n) \geq 0
	$$
	which is absurd  since $c<0$. 
	
On the other hand, using the fact that $(u_n)$ is bounded in $H_{rad}^{1}(\mathbb{R}^2)$, it follows from (\ref{EQ10N2}) and $(u_n)\subset S(a)$ that $(\lambda_n)$ is also a bounded sequence, then we can assume that for some subsequence $\lambda_n \to \lambda_a \in \mathbb{R}$. This limit together with (\ref{EQ10N2}) ensures that $u$ is a nontrivial solution of the equation
\begin{equation} \label{EQ2N2}
	-\Delta u -\mu|u|^{q-2}u-f(u)=\lambda_a u, \quad \mbox{in} \quad \mathbb{R}^2.
\end{equation}
By the Pohozaev's identity, we must have $\lambda_a<0$, because $u$ is nontrivial.
	
On the other hand,
$$
|\nabla u_n|_{2}^{2}-\lambda_n|u_n|_{2}^{2}=\mu\int_{\mathbb{R}^2}|u_n|^{q}\,dx+\int_{\mathbb{R}^2}f(u_n)u_n\,dx+o_n(1),
$$
or yet,
$$
|\nabla u_n|_{2}^{2}-\lambda_a|u_n|_{2}^{2}=\mu\int_{\mathbb{R}^2}|u_n|^{q}\,dx+\int_{\mathbb{R}^2}f(u_n)u_n\,dx+o_n(1).
$$
Recalling that by Lemma \ref{Convergencia em limitados1}
$$
\lim_{n \to +\infty}\int_{\mathbb{R}^2}f(u_n)u_n\,dx=\int_{\mathbb{R}^2}f(u)u\,dx,
$$
we derive that
$$
|\nabla u_n|_{2}^{2}-\lambda_a|u_n|_{2}^{2}=\mu\int_{\mathbb{R}^2}|u|^{q}\,dx+\int_{\mathbb{R}^2}f(u)u\,dx+o_n(1).
$$
By (\ref{EQ2N2}), we know that
$$
|\nabla u|_{2}^{2}-\lambda_a|u|_{2}^{2}=\mu\int_{\mathbb{R}^2}|u|^{q}\,dx+\int_{\mathbb{R}^2}f(u)u\,dx,
$$
consequently
$$
\lim_{n \to +\infty}(|\nabla u_n|_{2}^{2}-\lambda_a|u_n|_{2}^{2})=|\nabla u|_{2}^{2}-\lambda_a|u|_{2}^{2}.
$$
As $\lambda_a<0$, the last equality implies that
$$
u_n \to u \quad \mbox{in} \quad H_{rad}^{1}(\mathbb{R}^2),
$$
implying that $|u|_{2}=a$. This establishes the desired result.

\end{proof}

Now, for $\epsilon >0$, we define the set
$$
A=\{ u \in H_{rad}^{1}(\mathbb{R}^2)\cap S(a): J_T(u)\leq-\epsilon\}\subset H_{rad}^{1}(\mathbb{R}^2),
$$
which is a closed symmetric subset of $S(a)$, because $J_T$ is even and continuous.

\begin{lemma}\label{Lemma 4.3}
	Given $n \in \mathbb{N},$ there are $ \epsilon =\epsilon(n)>0$ and $\mu=\mu(n)>0$  such that $\mathcal{G}(A)\geq n,$ for any $ 0<\epsilon \leq\epsilon(n)$ and  $\mu \geq \mu_n$.
\end{lemma}

\begin{proof} Arguing as \cite{GP}, for each $n  \in \mathbb{N},$ let us consider an $n$ -dimensional $E_n \subset H_{rad}^{1}(\mathbb{R}^2) $ that has an orthogonal basis of the form
	$$
	\mathcal{B}=\{u_1,u_2,...,u_m\}
	$$
	such that $\displaystyle \int_{\mathbb{R}^2}\nabla u_j \nabla u_k\,dx=\int_{\mathbb{R}^2}u_ju_k\,dx=0$ for $j \not= k$ and $|\nabla u_j|_2 =\rho<\sqrt{1-a^2}$ and $|u_j|_2=a$, and so, $\|u_j\|=\sqrt{\rho^2 + a^2}$ for $j=1, 2, \cdots, n$. By $(f_3)$,
$$
J_T(v) \leq \frac{1}{2}\int_{\mathbb{R}^2} |\nabla v|^2 \,dx-\frac{\mu}{q}\int_{\mathbb{R}^2}|v|^q \,dx -\frac{1}{p}\int_{\mathbb{R}^2}|v|^p \,dx, \quad \forall v \in H^{1}(\mathbb{R}^2).
$$
Arguing as in the proof of Lemma \ref{Lemma 4.33}, we get
$$
J_T(v) \leq \frac{1}{2}\rho^2-\frac{\mu\beta_n}{q}\rho^{q} -C_1\alpha_n \rho^{p}, \quad \forall v \in \Upsilon_n,
$$
where
$$
\alpha_n=\inf\left\{ \int_{\mathbb{R}^2} |u|^{p}\, dx: u\in S(a/\rho) \cap E_n, |\nabla u|_2=1\right\} >0
$$
and	
$$
\beta_n=\inf\left\{ \int_{\mathbb{R}^2} |u|^{q}\, dx: u\in S(a/\rho) \cap E_n, |\nabla u|_2=1\right\} >0.
$$
From this, we can choose $\mu=\mu(n)>0$, $\epsilon=\epsilon(n)>0$ and $\rho< R_0$ such that
$$
J_T(v)\leq -\epsilon, \quad \forall v \in \Upsilon_n.
$$
Hence, $\Upsilon_n \subset A$ and $\mathcal{G} (A) \geq\mathcal{G}(\Upsilon_n)=n$.
	
\end{proof}

\subsection{Proof Theorem \ref{T2}}

The proof of Theorem \ref{T2} follows from Proposition \ref{Lemma 4.4}, because the critical points of $J_T$ were found in that proposition are in fact critical points of $J$, see Lemma \ref{Lemma 4.2N2}-(ii).

\noindent \textsc{Claudianor O. Alves } \\
Unidade Acad\^{e}mica de Matem\'atica\\
Universidade Federal de Campina Grande \\
Campina Grande, PB, CEP:58429-900, Brazil \\
\texttt{coalves@mat.ufcg.edu.br} \\
\noindent and \\
\noindent \textsc{Chao Ji} \\
Department of Mathematics\\
East China University of Science and Technology \\
Shanghai 200237, PR China \\
\texttt{jichao@ecust.edu.cn}\\
\noindent and \\
\noindent \textsc{Ol\'{i}mpio H. Miyagaki} \\
Departamento de Matem\'{a}tica \\
Universidade Federal de S\~ao Carlos \\
S\~ao Carlos,  SP, CEP:13565-905,  Brazil\\
\texttt{olimpio@ufscar.br}

\end{document}